\documentclass[a4paper,10pt]{amsart}
\usepackage[english]{babel}
\usepackage{amsmath,amssymb,amsthm}
\usepackage{multicol}
\usepackage[titletoc]{appendix}
\usepackage[mathscr]{eucal}

\usepackage[pdftex]{color}
\usepackage[bookmarks=true,hyperindex,pdftex,colorlinks, citecolor=blue,linkcolor=blue, urlcolor=blue]{hyperref}
\usepackage[dvipsnames]{xcolor}
 \usepackage{mathtools}

\usepackage{graphicx}

\usepackage[normalem]{ulem}

\theoremstyle{plain}
\newtheorem{theorem}{Theorem}[section]
\newtheorem{corollary}[theorem]{Corollary}
\newtheorem{proposition}[theorem]{Proposition}
\newtheorem{lemma}[theorem]{Lemma}

\theoremstyle{definition}
\newtheorem{remark}[theorem]{Remark}
\newtheorem*{remark-final}{Final remark}
\newtheorem{example}[theorem]{Example}
\newtheorem{definition}[theorem]{Definition}

 \DeclareMathOperator{\re}{Re}
 
 \DeclareMathOperator{\Id}{\mathrm{Id}}

\newcommand{\K}{\mathbb{K}}

\newcommand{\C}{\mathbb{C}}
\newcommand{\R}{\mathbb{R}}

\newcommand{\U}{\mathscr{U}}

\newcommand{\eps}{\varepsilon}

\renewcommand{\geq}{\geqslant}
\renewcommand{\leq}{\leqslant}

\begin{document}
\title{On Banach spaces whose group of isometries acts micro-transitively on the unit sphere}

\author[Cabello]{F\'{e}lix Cabello S\'anchez}
\address[Cabello]{Departamento de Matem\'{a}ticas and IMUEx, Universidad de Extremadura, 06071-Badajoz, Spain\newline
\href{http://orcid.org/}{ORCID: \texttt{0000-0003-0924-5189}  }}
\email{\texttt{fcabello@unex.es}}

\author[Dantas]{Sheldon Dantas}
\address[Dantas]{Department of Mathematics, Faculty of Electrical Engineering, Czech Technical University in Prague, Technick\'a 2, 166 27 Prague 6, Czech Republic
\newline
\href{http://orcid.org/0000-0001-8117-3760}{ORCID: \texttt{0000-0001-8117-3760}  }}
\email{\texttt{gildashe@fel.cvut.cz}}

\author[Kadets]{Vladimir Kadets}
\address[Kadets]{School of Mathematics and Computer Sciences, V.~N.~Karazin Kharkiv National University, pl.~Svobody~4, 61022~Kharkiv, Ukraine
\newline
\href{http://orcid.org/0000-0002-5606-2679}{ORCID: \texttt{0000-0002-5606-2679} }
}
\email{\texttt{v.kateds@karazin.ua}}

\author[Kim]{\\ Sun Kwang Kim}
\address[Kim]{Department of Mathematics, Chungbuk National University, 1 Chungdae-ro, Seowon-Gu, Cheongju, Chungbuk 28644, Republic of Korea \newline
	\href{http://orcid.org/0000-0002-9402-2002}{ORCID: \texttt{0000-0002-9402-2002}  }}
\email{\texttt{skk@chungbuk.ac.kr}}

\author[Lee]{Han Ju Lee}
\address[Lee]{Department of Mathematics Education, Dongguk University - Seoul, 04620 (Seoul), Republic of Korea \newline
\href{http://orcid.org/0000-0001-9523-2987}{ORCID: \texttt{0000-0001-9523-2987}  }
}
\email{\texttt{hanjulee@dongguk.edu}}

\author[Mart\'{\i}n]{Miguel Mart\'{\i}n}
\address[Mart\'{\i}n]{Departamento de An\'{a}lisis Matem\'{a}tico, Facultad de
 Ciencias, Universidad de Granada, 18071 Granada, Spain \newline
\href{http://orcid.org/0000-0003-4502-798X}{ORCID: \texttt{0000-0003-4502-798X} }
 }
\email{\texttt{mmartins@ugr.es}}

\begin{abstract}
We study Banach spaces whose group of isometries acts micro-transitively on the unit sphere. We introduce a weaker property, which one-complemented subspaces inherit, that we call uniform micro-semitransitivity. We prove a number of results about both micro-transitive and uniformly micro-semitransitive spaces, including that they are uniformly convex and uniformly smooth, and that they form a self-dual class. To this end, we relate the fact that the group of isometries acts micro-transitively with a property of operators called the pointwise Bishop-Phelps-Bollob\'{a}s property and use some known results on it. Besides, we show that if there is a non-Hilbertian non-separable Banach space with uniform micro-semitransitive (or micro-transitive) norm, then there is a non-Hilbertian separable one. Finally, we show that an $L_p(\mu)$ space is micro-transitive or uniformly micro-semitransitive only when $p=2$.
\end{abstract}

\date{June 21st, 2019}

\thanks{The first author was supported by DGICYT project MTM2016-6958-C2-1-P (Spain) and Consejer\'{\i}a de Educaci\'{o}n y Empleo, Junta de Extremadura program IB-16056. The second author was supported by the project OPVVV CAAS CZ.02.1.01/0.0/0.0/16\_019/0000778. The research of the third author is done in frames of Ukrainian Ministry of Science and Education Research Program 0118U002036, and it was partially supported by Spanish MINECO/FEDER projects MTM2015-65020-P and MTM2017-83262-C2-2-P. Fourth author was partially supported by Basic Science Research Program through the National Research Foundation of Korea(NRF) funded by the Ministry of Education, Science and Technology (NRF-2017R1C1B1002928). Fifth author was partially supported by Basic Science Research Program through the National Research Foundation of Korea (NRF) funded by the Ministry of Education, Science and Technology (NRF-2016R1D1A1B03934771). Sixth author partially supported by projects MTM2015-65020-P (MINECO/FEDER, UE) and PGC2018-093794-B-I00 (MCIU/AEI/FEDER, UE)}


\maketitle

\thispagestyle{plain}

\section{Introduction}

Let $G$ be a (Hausdorff) topological group with identity $e$ and let $T$ be a (Hausdorff) topological space. By an \emph{action} of $G$ on $T$ we mean a continuous map $G\times T \longrightarrow T$, $(g,x)\longmapsto gx$, such that $ex=x$ for every $x\in T$ and  $g(hx)=(gh)x$ for every $g,h\in G$ and every $x\in T$. The action of $G$ on $T$ is \emph{transitive} if the orbit of every element is the whole space, that is, if $Gx=T$ for every $x\in T$ (if $U\subset G$, $Ux \coloneqq \{gx\colon g\in U\}$). The action is \emph{micro-transitive} if for every $x\in T$ and every neighborhood $U$ of $e$ in $G$, the set $Ux$ is a neighborhood of $x$ in $T$; in this case, the orbits of the elements are open (see \cite[Lemma 1]{Ancel1987} or \cite[Assertion 1 on p.~106]{KozlovChatyrko}) and they form a partition of the space $T$ into disjoint open sets. Therefore, when $T$ is connected, micro-transitivity implies transitivity. Conversely, there is a famous result of E.~G.~Effros \cite{Effros} saying that if $G$ is Polish (i.e.,\ separable completely metrizable) acting transitively on a second category separable topological space, then it acts micro-transitively, see \cite{Ancel1987} or \cite{vanMill}. Micro-transitive actions are also called \emph{open} actions, as to be micro-transitive is equivalent to the fact that the application $g\longmapsto gx$ is open from $G$ to $T$ for every $x\in T$ \cite[Assertion 1 in p.~106]{KozlovChatyrko}. We refer the reader to already cited \cite{Ancel1987,KozlovChatyrko,vanMill} and references therein for more information and background on micro-transitive actions.

A famous open problem of S. Mazur in Banach space theory asks whether a separable
Banach space whose group of surjective isometries acts transitively on its unit sphere has to be
isometrically isomorphic to a Hilbert space; see \cite[p.~151]{Banach}. To simplify the notation, it is usually said that the norm of a Banach space is \emph{transitive} if its group of surjective isometries acts transitively on its unit sphere. We refer the reader to the classical book \cite{Rolewicz} and the expository papers \cite{BecerraRodriguez2002,Cabello-1997} for an extensive account on Mazur's problem, and to \cite{DilworthRandria, Ferenczi-Rosendal, Ferenczi-Rosendal-2017, Rambla} and references therein for more recent results. In the finite-dimensional setting, Mazur question was solved very early in the affirmative (\cite{AMU}; see \cite[Corollary 2.42]{BecerraRodriguez2002} for a contemporary proof). For non-separable Banach spaces, the answer to Mazur's problem is known to be negative as, for instance, there are non-separable $L_p(\mu)$ spaces whose norms are transitive for $1\leq p <\infty$; see \cite[Proposition 9.6.7]{Rolewicz}. Actually, ultrapowers of $L_p[0,1]$ do the job, use \cite[Example~2.13 and Proposition 2.19]{BecerraRodriguez2002}, for instance. Moreover, every Banach space can be isometrically regarded as a subspace of a suitable transitive Banach space (see \cite[Corollary 2.21]{BecerraRodriguez2002}) and every dual space is one-complemented in some transitive space (see \cite[Proposition 2.22]{BecerraRodriguez2002}) and, therefore, there exist transitive Banach spaces failing the approximation property (see \cite[Corollary 2.23]{BecerraRodriguez2002}).

In this paper we deal with the ``microscopic'' version of the transitivity: we say that the norm of a Banach space is \emph{micro-transitive} if its group of surjective isometries acts micro-transitively on its unit sphere. This is the case of the norm of a Hilbert space. Observe that due to the comments on the first paragraph, micro-transitive norms are transitive (we may also use \cite[Proposition 2.6]{BecerraRodriguez2002}, for instance). Moreover if the norm of a Banach space $X$ is micro-transitive, then actually a uniform version of the property holds (see Proposition \ref{prop-micro-implies-uniform}): for every $\eps>0$, there is $\delta>0$ such that if $x,y\in X$ with $\|x\|=\|y\|=1$ satisfy $\|x-y\|<\delta$, then there is a linear surjective isometry $T\colon X\longrightarrow X$ such that $Tx=y$ and $\|T-\Id\|<\eps$, where $\Id$ is the identity operator on $X$. Due to this fact, we also introduce and study a weaker property, which we will call \emph{uniform micro-semitransitivity} (see Definition \ref{def:uniform-micro-semitransitivity}) relaxing the requirement of $T$ being an isometry to only require $\|T\|= 1$. Contrary to what happens with most classical transitivity properties, uniform micro-semitransitivity is inherited by 1-complemented subspaces.

We have devoted a considerable effort to prove that both micro-transitivity and uniform micro-semitransitivity are somehow ``separably determined'': if $X$ is a micro-transitive (respectively, uniform micro-semitransitive), then every separable subspace of $X$ is contained in a separable micro-transitive (respectively, uniform micro-semitransitive) subspace. This is also very different than the usual behaviour of transitive properties.

Further, we relate uniform micro-semitransitivity with a property called pointwise Bishop-Phelps-Bollob\'{a}s property (see Definition \ref{def:pointwiseBPBp}) which has something to do with norm-attaining operators. With all of these, we will be able to show that
every uniformly micro-semitransitive space is both uniformly smooth and uniformly convex, with modulus of convexity of power type.

Finally, we show that an $L_p(\mu)$-space is not uniformly micro-semitransitive unless it is Hilbertian. In the way of proving this result, we show that both micro-transitivity and uniform micro-semitransitivity pass from a norm to its dual.

\subsection*{Notation and terminology}\label{subsection-notation}
Given a Banach space $X$ over the field $\K$ ($=\R$ or $\C$), we write $B_X$ and $S_X$ to denote, respectively, the closed unit ball and the closed unit sphere of $X$; $X^*$ is the dual space of $X$ and $\mathcal{L}(X)$ is the space of bounded linear operators from  $X$ to itself. If $Y$ is another Banach space, $\mathcal{L}(X,Y)$ is the Banach space of all bounded linear operators from $X$ to $Y$.

A Banach space $X$ is said to be \emph{smooth} if the norm is G\^{a}teaux differentiable at every non-zero point of $X$ or, equivalently, if for every $x\in X\setminus \{0\}$ there is only one $x^*\in S_{X^*}$ such that $x^*(x)=\|x\|$. A space $X$ is said to be \emph{uniformly smooth} if its norm is Fr\'{e}chet differentiable uniformly on $S_X$. The \emph{modulus of convexity} of a Banach space $X$ is given by
$$
\delta_X(\eps)=\inf\left\{1-\left\|\frac{x+y}{2}\right\|\colon x,y\in B_X,\, \|x-y\|\geq \eps\right\} \qquad (0<\eps\leq 2).
$$
The space $X$ is said to be \emph{uniformly convex} if $\delta_X(\eps)>0$ for all $\eps\in (0,2]$. Recall that a Banach space $X$ is uniformly convex if and only if $X^*$ is uniformly smooth. We refer the reader to \cite[\S 9]{FHHMPZ} and \cite{Diestel} for background on uniform convexity and uniform smoothness.

\section{Micro-transitivity and related properties}
Our aim here is to present the announced definition of uniform micro-semitransitivity and some preliminary results, to prove that uniform micro-semitransitivity and micro-transitivity are somehow separably determined,  to recall some known results on the pointwise Bishop-Phelps-Bollob\'{a}s property, and to relate both properties.

\subsection{Uniform micro-semitransitivity}\label{subsect-uniform-micro-semitran}
We start this subsection by showing that a Banach space with micro-transitive norm actually fulfills a uniform version of the property.

\begin{proposition}\label{prop-micro-implies-uniform}
Let $X$ be a Banach space with micro-transitive norm. Then, there is a function $\beta\colon(0,2)\longrightarrow \R^+$ such that if $x,y\in S_X$ satisfy $\|x-y\|<\beta(\eps)$, then there is a surjective isometry $T\in \mathcal{L}(X)$ satisfying that $Tx=y$ and $\|T-\Id\|<\eps$.
\end{proposition}

\begin{proof}
Let $0<\eps<2$ and write $U=\{S\in \mathcal{L}(X)\colon S \text{ onto isometry},\ \|S-\Id\|<\eps\}$. Fix $x_0\in S_X$ and use the micro-transitivity of the norm to find $\delta>0$ such that
$$
\{x\in S_X\colon \|x-x_0\|<\delta\}\subseteq Ux_0.
$$
Now, given $x,y\in S_X$ with $\|x-y\|<\delta$, we find an onto isometry $R\in \mathcal{L}(X)$ such that $Rx=x_0$ (this is possible as micro-transitivity implies transitivity) and observe that $Ry\in S_X$ satisfies that
$$
\|x_0-Ry\|=\|Rx-Ry\|=\|x-y\|<\delta.
$$
Then, there is $S\in U$ such that $Sx_0=S(Rx)=Ry$. Now, the operator $T=R^{-1}SR$ is an onto isometry, $Tx=R^{-1}Sx_0=y$ and
\[
\|T-\Id\|=\|R^{-1}SR-R^{-1}R\|=\|S-\Id\|<\eps.\qedhere
\]
\end{proof}

Hilbert spaces are the touchstone of all transitivity problems. Let us take a look at the action of their isometry groups. Let $H$ be a real Hilbert space and $x,y\in S_H$. Then there is a surjective isometry $T$ of $H$ such that $y=Tx$ and $\|T-\Id\|=\|y-x\|$. To see this we may assume that $x$ and $y$ are linearly independent. Let $E$ be the two-dimensional subspace of $H$ spanned by $x$ and $y$ and let $R$ be the unique isometry on $E$ that rotates $x$ onto $y$ and has minimal angle. Let $T$ be the isometry that extends $R$ as the identity on the orthogonal complement of $E$. It is pretty clear that $\|T-\Id\|=\|y-x\|$ and also that $T-\Id$ has rank two. Hence real Hilbert spaces are micro-transitive with $\beta(\eps)=\eps$ for every $\eps\in (0,2)$; see \cite[Lemma~2.2]{AMS} for complex Hilbert spaces.

Proposition \ref{prop-micro-implies-uniform} motivates us to introduce the following definition.

\begin{definition}\label{def:uniform-micro-semitransitivity}
Let $X$ be a Banach space. We say that $X$ (or the norm of $X$) is \emph{uniformly micro-semitransitive} if there exists a function $\beta\colon(0,2)\longrightarrow \R^+$ such that whenever $x,y\in S_X$ satisfies $\|x-y\|<\beta(\eps)$, then there is $T\in \mathcal{L}(X)$ with $\|T\|=1$ satisfying $$Tx=y \qquad \text{and} \qquad \|T-\Id\|<\eps.$$
\end{definition}

Observe that Proposition \ref{prop-micro-implies-uniform} implies that micro-transitive norms are uniformly micro-semitransitive. The following observation is immediate.

\begin{remark}\label{Remark-one-complemented}
If the norm of a Banach space $X$ is uniformly micro-semitransitive and $Y$ is a one-complemented subspace of $X$, then the norm of $Y$ is uniformly micro-semitransitive.
\end{remark}

The name micro-semitransitivity comes from the fact that to get a micro-transitive norm from Definition \ref{def:uniform-micro-semitransitivity}, one just needs to further impose that $T$ is invertible (which is free for $0<\eps<1$) and that $\|T^{-1}\|=1$, so we need a half of the requirements.
It is worth to remark that if the norm of $X$ is uniformly micro-semitransitive, we may further get that the norm is ``transitive by norm-one isomorphisms'':

\begin{remark}\label{remark-exists-isomorphism}
{\slshape If the norm of a Banach space $X$ is uniformly micro-semitransitive, then given $x,y\in S_X$, there is an isomorphism $T\in \mathcal{L}(X)$ with $\|T\|=1$ such that $Tx=y$.}\ Indeed, we may find $x_1,\ldots,x_n\in S_X$ such that $x_1=x$, $x_n=y$, and $\|x_{i+1}-x_i\|<\beta(1/2)$ for $i=1,\ldots,n-1$. Then, by hypothesis, there are $T_1,\ldots,T_{n-1}\in \mathcal{L}(X)$ with $\|T_i\|=1$ satisfying that $\|T_i-\Id\|<1/2$ and $T_i(x_i)=x_{i+1}$ for $i=1,\ldots,n-1$. As all $T_i$ are isomorphisms by the Neumann series for $i=1,\ldots,n-1$, so is $T=T_{n-1}\circ\cdots\circ T_1$. Further, $\|T\|\leq 1$ and $Tx=y$ (so, in particular, $\|T\|=1$).
\end{remark}

Let us observe that, while it is conceivable that the property of the remark above (i.e.\ ``transitivity by norm-one isomorphisms'') characterizes Hilbert spaces among separable ones, it is immediate that all Banach spaces fulfill the following property: {\slshape given $x,y\in S_X$, there is $T\in \mathcal{L}(X)$ with $\|T\|=1$ such that $Tx=y$.} Indeed, consider $x^*\in S_{X^*}$ such that $x^*(x)=1$ and define the rank-one operator $T\in \mathcal{L}(X)$ by $Tz=x^*(z)y$ for all $z\in X$; then $\|T\|=1$ and $Tx=y$.
\medskip

\subsection{Micro-transitivity, separability, and completeness}
Mazur rotation problem is arguably very difficult: one can hardly imagine a way to produce a counterexample since completeness tends to interfere with separability in transitive spaces, and vice-versa. In the opposite direction, separability and completeness seem to be superfluous. See, however, \cite[Theorem~2]{Ferenczi-Rosendal-2017}.

All these difficulties disappear in our micro-transitivity setting. This is so because both  micro-transitivity and  uniform micro-semitransitivity are separably determined and they do not really depend on completeness, as we will see very soon.
\medskip

For a normed space $X$ we will use the name ``contractive automorphism of $X$'' for any linear operator $T \colon X \longrightarrow X$ with $\|T\| \leq 1$. Let $X$ be a normed space, $U$ a subset of the semigroup of contractive automorphisms of $X$ and $E$ a subset of the unit sphere of $X$, not necessarily $U$-invariant. Let us say, slightly bending the rules, that $U$ acts uniformly micro-semitransitively on $E$ if, for each $\eps>0$, there exists $\beta(\eps)>0$ such that, whenever $x,y\in E$ satisfy $\|x-y\|<\beta(\eps)$, there is $T\in U$ satisfying $Tx=y$ and $\|T-\Id\|<\eps.$

\begin{lemma}\label{lem:dense}
A Banach space is uniformly micro-semitransitive (respectively, micro-transitive) if and only if the semigroup of contractive automorphisms (respectively, the isometry group) acts uniformly micro-semitransitively on some dense subset of the unit sphere.
\end{lemma}

\begin{proof}
We only write the proof for uniform micro-semitransitivity. The other case follows the same lines.

Assume that $D$ is dense in $S_X$ and that $U$, the semigroup of all contractive automorphisms of $X$, acts uniformly micro-semitransitively on $D$.
\smallskip

\noindent\emph{Step 1.} For every
$\eps>0$, there exists $\delta(\eps)>0$ so that, if $x\in D$ and  $y\in S_X$ satisfy $\|x-y\|<\delta(\eps)$, then there is $T\in U$ such that $Tx=y$ and $\|T-\Id\|<\eps.$
\smallskip

For each $n\geq 1$ let $\delta_n=\beta(2^{-n})$ be the number provided by the definition. Needless to say, $\delta_n\leq 2^{-n}$. Pick $\eps>0$ and then $k$ so large that $\sum_{n\geq k}2^{-n}<\eps$.

Set $\delta=\beta(\eps)$ and assume that $\|x - y\| < \delta$. Choose a sequence of points in $D\cap B(x,\delta)$, labeled as $(y_n)_{n\geq k}$, so that $\|y_n-y\|<\delta_{n+1}/2$ and $\|y_n-y_{n+1}\|<\delta_{n+1}$. After that, take a sequence of contractive automorphisms $(T_n)_{n\geq k}$ such that
\begin{itemize}
\item $y_k=T_k(x)$ and $\|T_k-\Id\|<\eps$.
\item For $n>k$ one has $y_n=T_n(y_{n-1})$ and $\|T_n-\Id\|<2^{-n}$.
\end{itemize}
Let us see that the operator $T$ defined by
$$
T(z)=\lim_{m\to\infty}\bigl[T_m\circ\dots \circ T_{k+1}\circ T_k\bigr] (z) \qquad (z\in X),
$$
does the trick. Of course we have to check that the definition makes sense.

First of all, note that for $m>n\geq k$ one has the obvious estimate
$$
\bigl\|\bigl[T_m\circ\dots \circ T_{n}\bigr]-\Id\bigr\|\leq \sum_{i=n}^m\|T_i-\Id\|.
$$
This shows that the sequence defining $T(z)$ converges for all $z$, that $T$ is a contractive operator, and also that
$$
\|T-\Id\|\leq \eps+\sum_{i>k}2^{-i} \leq 2\eps.
$$
In particular, $T$ is onto for $0<\eps<{1\over 2}$ and since the equality $y=Tx$ is trivial, this concludes the proof of Step~1.
\smallskip

\noindent\emph{Step 2.} For every
$\eps>0$, there exists $\eta>0$, depending only on $\eps$, so that, if $x, y\in S_X$ satisfy $\|x-y\|<\eta$, then there is $T\in U$ such that $Tx=y$ and $\|T-\Id\|<\eps.$
\smallskip

This is just a rewording of Definition~\ref{def:uniform-micro-semitransitivity}. To prove it we proceed as before, using Step 1.

For each $n\geq 1$, set $\delta_n=\delta(2^{-n})$, where $\delta(\cdot)$ is the function provided by Step 1. Fix $\eps>0$ and take $k$ so large that $\sum_{n\geq k}2^{-n}<\eps$.

Set $\eta=\delta(\eps)$. Now assume $\|x-y\|<\eta$, with $x,y\in S_X$.
Choose a sequence of points in $D\cap B(y,\eta)$, labeled as $(x_n)_{n\geq k}$ so that $\|x_n-x\|<\delta_{n}/2$. In particular, $\|x_n-x_{n+1}\|<\delta_{n}$.  Take then a sequence of contractive automorphisms $(L_n)_{n\geq k}$ such that
\begin{itemize}
\item $L_k(x_k)=y$ and $\|L_k-\Id\|<\eps$.
\item For $n>k$ one has $L_n(x_{n+1})=x_n$ and $\|L_n-\Id\|<2^{-n}$.
\end{itemize}
Now, we consider the operator $L$ defined as
$$
L(z)=\lim_{n\to\infty}\bigl[L_k\circ L_{k+1}\dots\circ L_n\bigr] (z) \qquad (z\in X).
$$
Please note that this time we compose \emph{on the right}.

Using the estimate $\bigl\|\bigl[L_n\circ\dots \circ L_m\bigr]-\Id\bigr\|\leq \sum_{i=n}^m\|L_i-\Id\|$, it is easy to see that $L(z)$ is correctly defined, that $\|L\|\leq 1$, and also that $\|L-\Id\|<2\eps$. In particular $L$ belongs to $U$ provided $\eps<{1\over 2}$. It only remains to check that $y=Lx$, which is obvious after realizing the following facts:
\begin{itemize}
\item $\| L_k\circ\dots \circ L_n- L\|\to 0$ as $n\to\infty$.
\item For each $n>k$ one has $y= L_k\circ\dots \circ L_n(x_{n+1})$.
\item $\|x_n-x\|\to 0$ as $n\to\infty$.\vspace{-18pt}
\end{itemize}
\end{proof}

Thus, the completion of a uniformly micro-semitransitive \emph{normed} space is then a uniformly micro-semitransitive Banach space. The following result shows that uniformly micro-semitransitive spaces are ``somehow separably determined''.

 \begin{theorem}\label{thm:separable}
Let $X$ be a uniformly micro-semitransitive (respectively, micro-transitive) Banach space and $Y$ a separable subspace of $X$. Then there is another separable subspace of $X$ which is uniformly micro-semitransitive (respectively, micro-transitive) and contains $Y$.
 \end{theorem}

\begin{proof}
As before, we do the proof for uniform micro-semitransitivity.
The other case requires only minor changes.

Let $\beta\colon(0,2)\longrightarrow \R^+$ be any function witnessing that
$X$ is uniformly micro-semitransitive. We are going to construct three sequences
$(Y_n)_{n\geq 0},  (D_n)_{n\geq 0}, (G_n)_{n\geq 0}$ having the following properties:
\begin{enumerate}
\item $(Y_n)_{n\geq 0}$ is an increasing sequence of separable subspaces of $X$, with $Y_0=Y$.
\item Every $G_n$ is a countable semigroup of automorphisms of $X$.
\item $G_{n}\subset G_{n+1}$ for every $n\geq 1$.
\item Every operator in $G_n$ leaves $Y_n$ invariant.
\item $D_n$ is a countable, dense subset of the unit sphere of $Y_n$.
\item $D_{n}\subset D_{n+1}$ for every $n\geq 1$.
\item Given $\eps>0$, if $x,y\in D_n$ are such that $\|x-y\|<\beta(\eps)$, there is  a contractive $T\in G_{n+1}$ such that $y=Tx$ and $\|T-\Id\|<\eps$.
\end{enumerate}
We proceed by induction on $n\geq 0$. In the initial step we put $Y_0=Y$, the set $D_0$  is any countable, dense subset of the unit sphere of $Y$ and $G_0=\{\Id\}$.

The induction step is as follows. Denote the semigroup of contractive automorphisms of $X$ by $U$. Assume $Y_n, D_n, G_n$ have already been defined.
For each pair $x,y\in D_n$ we select $T_k\in U$ such that $y=T_kx$ for all $k$, and
$$
\|T_k-\Id\|<\inf\{\|T -\Id\|\colon y=Tx, T\in U\}+k^{-1}.
$$
Observe that we need only countably many automorphisms. Let $U_{n+1}$ be the (countable) semigroup generated by all these automorphisms and those of $G_n$.
Then we define $Y_{n+1}$ as the closed subspace of $Y$ spanned by the set
$$
\bigcup_{T\in G_{n+1}} T[Y_n].
$$
Clearly, $Y_{n+1}$ is separable and contains $Y_n$ since $\Id\in G_{n+1}$. By the very definition, every operator in $G_{n+1}$ leaves $Y_{n+1}$ invariant. Taking any countable dense subset $D_{n+1}$ of the unit sphere of $Y_{n+1}$ containing $D_n$ we close the circle.
\smallskip

To end, consider the ``limit objects''
$$
Y_\infty=\overline{\bigcup_{n\geq 0} Y_n},\qquad
G_\infty={\bigcup_{n\geq 0} G_n},\qquad
D_\infty={\bigcup_{n\geq 0} D_n}.
$$
Clearly, $Y_\infty$ is separable and contains $Y$ and $D_\infty$ is dense in the unit sphere of $Y_\infty$. Moreover, and this is the key point, every operator $T\in G_\infty$ leaves $Y_\infty$ invariant and so $T$ restricts to an automorphism of $Y_\infty$, with $\|T|_{Y_\infty}-\Id_{Y_\infty}\|\leq \|T-\Id_X\|$.

Finally, condition (7) guarantees that the semigroup of contractive automorphisms of $Y_\infty$ acts uniformly micro-semitransitively on $D_\infty$ and Lemma~\ref{lem:dense} ends the proof.
\end{proof}

The immediate consequence of the preceding Theorem~\ref{thm:separable} is that if every separable transitive Banach space is isometric (or isomorphic) to a Hilbert space, then
every micro-transitive Banach space is isometric (or isomorphic) to a Hilbert space since the parallelogram law and Kwapien's theorem \cite{Kw} say that  a Banach space is  isometric (or isomorphic) to a Hilbert space if and only if so are its separable subspaces.

\subsection{Ultrapowers}
Ultraproducts (and ultrapowers) provide a quite useful technique to deal with transitivity questions. Here we will only recall a few definitions, mainly to fix the notations. The interested reader is referred to Sims' notes \cite{Sims} or Heinrich's classic \cite{Heinrich} for two rather different, complementary expositions.

Let $X$ be a Banach space, let $I$ be a non-empty set, and let $\U$ be a nontrivial ultrafilter on $I$. The space
of bounded families $ \ell_\infty(I, X)$ endowed
with the supremum norm is a Banach space, and the space $ c_0^\U(X)\coloneqq
\{(x_i) \in \ell_\infty(I,X)\colon \lim_{\U} \|x_i\|=0\} $ is a
closed subspace of $\ell_\infty(I, X)$. The \emph{ultrapower} of the
space $X$ following $\U$ is defined as the quotient space
$$
[X]_\U = {\ell_\infty(I, X)}/{c_0^\U(X)},
$$
with the quotient norm. We denote by $[(x_i)]$ the element of $[X]_\U$ which has the family $(x_i)$ as a representative. It is easy to see that $ \|[(x_i)]\| = \lim_{\U} \|x_i\|. $  If $T_i\colon X\longrightarrow Y$ , $i\in I$ is a
uniformly bounded family of operators, the ultraproduct operator
$[T_i]_\U\colon [X]_\U\longrightarrow [Y]_\U$ is given by $[T_i]_\U[(x_i)]=
[T_i(x_i)]$. Quite clearly, $ \|[T_i]_\U\|= \lim_{\U}\|T_i\|.$ Also,  $[T_i]_\U$ is a surjective isometry when all the $T_i$'s are.

Regarding the subject of this paper, the most salient feature of ultrapowers is the following result whose proof is nearly trivial.

\begin{lemma}\label{lemma:ultrapowers}
Ultrapowers preserve both micro-transitivity and uniform micro-semitransitivity.
\end{lemma}

For later use, it is worth to mention that the modulus of convexity of $X$ is the same as that of $[X]_\U$ \cite[p. 184]{KirkSims}. Also remark that
 $X$ is uniformly convex if and only if $[X]_\U$ is strictly convex. Indeed, if $X$ is not uniformly convex, there exist $\eps>0$ and two sequences $(x_i)$, $(y_i)$ in $S_X$ such that $\inf_{i\in \mathbb{N}}\|x_i-y_i\|\geq \eps$ and $\lim_i\left\| \frac{x_i+y_i}2\right\| =1$. This means that $[X]_\U$ is not strictly convex.

\subsection{Definition and some known results on the pointwise Bishop-Phelps-Bollob\'{a}s property} \label{subsec-BPBp}

Trying to extend to operators the improvement given by Bollob\'{a}s of the classical Bishop-Phelps theorem on the denseness of norm-attaining functionals, M.~D.~Acosta, R.~Aron, D.~Garc\'{\i}a, and M.~Maestre \cite{AAGM} introduced in 2008 the following property.

\begin{definition}[\textrm{\cite{AAGM}}]\label{def:BPBp}
A pair $(X,Y)$ of Banach spaces has the \emph{Bishop-Phelps-Bollob\'{a}s property} ({\it BPBp}, for short) if for every $\eps>0$, there exists $\eta(\eps)>0$ such that whenever $T\in \mathcal{L}(X,Y)$ with $\|T\|=1$ and $x_0\in S_X$ satisfy $\|Tx_0\|>1-\eta(\eps)$, there are $S\in \mathcal{L}(X,Y)$ with $\|S\| = 1$ and $x\in S_X$ such that $\|S\|=\|Sx\|=1$, $\|x_0-x\|<\eps$, and $\|S-T\|<\eps$.
\end{definition}
For  background on the BPBp, we refer the reader to the papers \cite{ABGM, AMS, ACKLM, CGKS, Cho-Choi,KL} and references therein. The next result summarizes the results of the BPBp that we will use in this paper.

\begin{proposition}[\textrm{\cite{AAGM,ABGM,KL}}]\label{prop:results-BPBp}
Let $X$ be a Banach space.
\begin{enumerate}
  \item[(a)] (Bishop-Phelps-Bollob\'{a}s theorem) The pair $(X,\K)$ has the BPBp.
  \item[(b)] If $X$ is uniformly convex, then the pair $(X,Y)$ has the BPBp for every Banach space $Y$. Moreover, the function $\eps\longmapsto \eta(\eps)$ in Definition \ref{def:BPBp} depends only on the modulus of convexity of $X$.
\end{enumerate}
\end{proposition}

If, in the definition of BPBp, the point where the new operator $S$ attains its norm is the same that the point where the operator $T$ almost attains its norm, we get the following stronger version of the Bishop-Phelps-Bollob\'{a}s property, introduced in \cite{DKL} and further developed in \cite{DaKaKiLeMa}.

\begin{definition}[\textrm{\cite{DKL}}]\label{def:pointwiseBPBp}
We say that a pair $(X, Y)$ of Banach spaces has the \emph{pointwise Bishop-Phelps-Bollob\'as property} (\emph{pointwise BPB property}, for short) if given $\eps > 0$, there exists $\tilde\eta(\eps) > 0$ such that whenever $T \in \mathcal{L}(X, Y)$ with $\|T\| = 1$ and $x_0 \in S_X$ satisfy
	\begin{equation*}
	\|T(x_0)\| > 1 - \tilde\eta(\eps),	
	\end{equation*}
there is $S \in \mathcal{L}(X, Y)$ with $\|S\| = 1$ such that
\begin{equation*}
\|S(x_0)\| = 1 \qquad \text{and} \qquad \|S - T\| < \eps.	
\end{equation*}
\end{definition}

The following are pairs of Banach spaces with the pointwise BPB property: $(H,Y)$ for a Hilbert space $H$ and every Banach space $Y$ \cite[Theorem 2.5]{DKL}, and $(X,C(K))$ for every uniformly smooth Banach space $X$ and every Hausdorff compact topological space $K$ \cite[Corollary 2.8]{DKL}. The pointwise BPB property is actually stronger than the BPBp, as the results from \cite{DaKaKiLeMa,DKL} show. We include here an omnibus result containing the main properties which we will use in this paper.

\begin{proposition}[\textrm{\cite{DaKaKiLeMa,DKL}}]
\label{omnibus}
Let $X$ be a Banach space.
\begin{enumerate}
  \item[(a)] If there is a nontrivial Banach space $Y$ such that the pair $(X,Y)$ has the pointwise BPB property, then $X$ is uniformly smooth.
  \item[(b)] If $(X,Y)$ has the pointwise BPB property for every Banach space $Y$, then $X$ is uniformly convex. Moreover, we have that
\[
\delta_{X}(\eps) \geq C\,\eps^q\qquad (0<\eps<2)
\]
for suitable $2\leq q<\infty$ and $C>0$.

  \item[(c)] If $X$ is isomorphic to a Hilbert space and $(X,Y)$ has the pointwise BPB property for every Banach space $Y$,  then $X$ has an optimal modulus of convexity, that is, there exists $C>0$ such that
$$
\delta_X(\eps)\geq C\,\eps^2 \qquad (0<\eps<2).
$$
Moreover, the constant $C$ depends only on the infimum over $Y$ of the collection of functions $\tilde{\eta}$ in Definition \ref{def:pointwiseBPBp} for all the pairs $(X,Y)$ (which is known to be always positive) and the Banach-Mazur distance from $X$ to the Hilbert space.
\end{enumerate}
\end{proposition}

\subsection{The relationship}\label{subsect:relationship}

As we commented before, it is shown in \cite[Theorem 2.5]{DKL} that a pair $(H,Y)$ where $H$ is a Hilbert space and $Y$ is arbitrary has the pointwise BPB property. To prove this result, it is used the fact that the norm of a Hilbert space is (uniformly) micro-transitive \cite[Lemma~2.2]{AMS}. Next, we prove that the uniform micro-semitransitivity of the norm allows to convert the BPBp to its pointwise version. This is, therefore, an abstract version of \cite[Theorem 2.5]{DKL}.

\begin{proposition}\label{prop2}
Let $X$ and $Y$ be Banach spaces. Suppose that the norm of $X$ is uniformly micro-semitransitive (in particular, if the norm is micro-transitive) and that the pair $(X, Y)$ has the BPBp. Then, the pair $(X, Y)$ has the pointwise BPB property.
\end{proposition}

\begin{proof}
Suppose that $(X, Y)$ has the BPBp with a function $\eps\longmapsto\bar\eta(\eps)$ and that the norm of $X$ is uniformly micro-semitransitive with the function $\eps\longmapsto \beta(\eps)$. Fix $\eps>0$ and suppose that $T \in \mathcal{L}(X, Y)$ with $\|T\| = 1$ and $x_0 \in S_X$ satisfy
\begin{equation*}
\|T(x_0)\| > 1 - \bar\eta\bigl(\beta(\eps/2)\bigr).
\end{equation*}
Then, by the BPBp of $(X,Y)$, there are  $\widetilde{S} \in \mathcal{L}(X, Y)$ with $\|\widetilde{S}\| = 1$ and $\widetilde{x}_0 \in S_X$ satisfying that
\begin{equation*}
\|\widetilde{S}(\widetilde{x}_0)\| = 1, \quad \|\widetilde{S} - T\| < \beta(\eps/2)\leq \frac{\eps}{2}, \quad \text{and} \quad \|x_0 - \widetilde{x}_0\| < \beta(\eps/2).
\end{equation*}
By the uniform micro-semitransitivity of the norm of $X$, there exists $R\in \mathcal{L}(X)$ with $\|R\| = 1$ such that
\begin{equation*}
R(x_0) = \widetilde{x}_0 \qquad \text{and} \qquad \|R - \Id\| < \frac{\eps}{2}.
\end{equation*}
Define $S   \coloneqq \widetilde{S} \circ R\in \mathcal{L}(X,Y)$. Then $\|S\| \leq 1$ and
\begin{equation*}
\|S(x_0)\| = \| \widetilde{S}(R(x_0))\| = \| \widetilde{S} (\widetilde{x}_0)\| = 1.
\end{equation*}
So, $\|S\| = \|S(x_0)\| = 1$. Moreover,
\begin{equation*}
\|S - T\| \leq \| \widetilde{S} \circ R - \widetilde{S}\| + \| \widetilde{S} - T\| \leq \| R - \Id\| + \frac{\eps}{2} < \eps.
\end{equation*}
This proves that the pair $(X, Y)$ has the pointwise BPB property as desired.
\end{proof}

As a consequence, we get again Theorem 2.5 of \cite{DKL} saying that every pair $(H,Y)$ where $H$ is a Hilbert space and $Y$ is a Banach space has the pointwise BPB property. But focusing on the uniform micro-semitransitivity, we get the following consequence.

\begin{corollary}\label{corollary-quasi-transitive=>unif-smooth}
Every Banach space whose norm is uniformly micro-semitransitive (in particular, if it is micro-transitive) is both uniformly smooth and uniformly convex.
\end{corollary}

\begin{proof}
Assume $X$ is uniformly micro-semitransitive.
By the Bishop-Phelps-Bollob\'as theorem (see Proposition \ref{prop:results-BPBp}.a), the pair $(X, \mathbb{K})$ has the BPBp. If the norm of $X$ is uniformly micro-semitransitive, then $(X, \mathbb{K})$ has the pointwise BPB property by Proposition \ref{prop2}. Hence, $X$ is uniformly smooth by Proposition \ref{omnibus}.a.

To prove that $X$ is uniformly convex it clearly suffices to show that the ultrapowers of $X$ are strictly convex. Note that a  uniformly micro-semitransitive space is strictly convex if and only if its unit sphere has an extreme point: this follows from Remark~\ref{remark-exists-isomorphism}, since a norm-one isomorphism cannot send a non-extreme point of the unit sphere into an extreme point of the unit sphere.

Now, if $X_\U$ is any ultrapower of $X$, then we know from Lemma \ref{lemma:ultrapowers} that it is uniformly micro-semitransitive, hence uniformly smooth for the first part of the proof, hence reflexive, and thus the unit ball of $X_\U$ has (many!) extreme points. By the preceding remark, $X_\U$ is strictly convex and so $X$ is uniformly convex.
\end{proof}

Another consequence of Proposition \ref{prop2} and the results presented in subsection \ref{subsec-BPBp} is the following.

\begin{corollary}\label{cor-charact-BPBp-BPBpp}
Let $X$ be a Banach space whose norm is uniformly micro-semitransitive. Then:
\begin{enumerate}
  \item[(a)] $(X,Y)$ has the BPBp for every Banach space $Y$;
  \item[(b)] Actually, $(X,Y)$ has the pointwise BPB property for every Banach space $Y$;
  \item[(c)] There exist $2\leq q<\infty$ and $C>0$ such that $\delta_{X}(\eps) \geq C\,\eps^q$ for every $0<\eps<2$.
\end{enumerate}
If, moreover, $X$ is isomorphic to a Hilbert space, then there exists $C>0$ such that
$$
\delta_X(\eps)\geq C\,\eps^2 \qquad (0<\eps<2),
$$
where the constant $C$ depends only on the modulus of convexity of $X$, on the function $\beta(\cdot)$ from the definition of uniform micro-transitivity, and on the Banach-Mazur distance from $X$ to the Hilbert space.
\end{corollary}

\begin{proof}
As $X$ is uniformly convex by Corollary \ref{corollary-quasi-transitive=>unif-smooth}, (a) follows from Proposition \ref{prop:results-BPBp}.b. Assertion (b) follows from (a) by using Proposition \ref{prop2}. Finally, (c) follows from (b) by using Proposition \ref{omnibus}.b.

The ``moreover'' part follows from Proposition \ref{omnibus}.d. Indeed, suppose that $X$ is uniformly convex and let $Y$ be an arbitrary Banach space. By Proposition \ref{prop:results-BPBp}.b, the function $\eps\longmapsto \eta(\eps)$ from the definition of the BPBp for the pair $(X,Y)$ only depends on the modulus of convexity of $X$; then, by the proof of Proposition \ref{prop2}, the function $\eps\longmapsto\tilde{\eta}(\eps)$ from the definition of the pointwise BPB property for $(X,Y)$ only depends on the modulus of convexity of $X$ and the function giving the uniform micro-semitransitivity of $X$. With this in mind, the result follows immediately from Proposition \ref{omnibus}.d.
\end{proof}

\section{Further results}\label{sec-main}
Our first aim here is to prove the following result.

\begin{theorem}\label{theorem:noLp-microtransitive}
The norm of $L_p(\mu)$ is uniformly micro-semitransitive (or micro-transitive) if and only if it is a Hilbertian norm. That is, $p=2$ or the space is one-dimensional.
\end{theorem}

Note that the above result shows that transitivity of a norm does not imply micro-transitivity and so, Effros' theorem can not be extended to the group of isometries of arbitrary Banach spaces. For possible extensions of Effros' result to some non-separable settings, we refer the reader to \cite{Ostaszewski}.

\begin{example}
{\slshape There are (non-separable) Banach spaces whose norm is transitive but not micro-transitive}.\ Indeed, it is known that for $1\leq p <\infty$, there are (non-separable) $L_p(\mu)$ spaces whose standard norms are transitive \cite[Proposition 9.6.7]{Rolewicz}, but they are not micro-transitive by Theorem \ref{theorem:noLp-microtransitive} except for $p= 2$.
\end{example}

To provide the proof of Theorem \ref{theorem:noLp-microtransitive} we need to state the following particular case.

\begin{example}\label{example-ellp2-direct}
For $1<p<\infty$ and $p\neq 2$, the norm of  $X=\ell_p^{(2)}$ is not uniformly micro-semitransitive.
\end{example}

\begin{proof}
As usual, we denote $e_1 = (1, 0)$,  $e_2 = (0, 1)$. For $a_p = 2^{- 1/p}$, we consider $x_p = a_p (1, 1) \in S_X$.  Suppose first that $p>2$. Let us demonstrate that there is no isomorphism $T\in \mathcal{L}(X)$ with $\|T\|=1$ such that $T(e_1) = x_p$. This would give the result by using Remark \ref{remark-exists-isomorphism}.
Indeed, we fix an operator $T\in \mathcal{L}(X)$ with $\|T\|=1$ such that $T(e_1) = x_p$ and we write $ Te_2 = (u_1, u_2)$.  Then, for all $t \in (-1, 1)$ we have that
$$
\|T(e_1 + t e_2)\|^p \leq 1 + |t|^p,
$$
that is,
$$
(a_p + t u_1)^p + (a_p + t u_2)^p  \leq 1 + |t|^p
$$
and
$$
\bigl(1 + t\, u_1/a_p\bigr)^p + \bigl(1 + t\, u_2/a_p\bigr)^p  \leq 2 + 2 |t|^p.
$$
For small values of $\tau$, we have asymptotically that $(1 + \tau)^p = 1 + p \tau + \frac{p(p-1)}{2} \tau^2 + o(\tau^2)$, so
$$
pt \bigl(u_1/a_p + u_2/a_p\bigr)  +  \frac{p(p-1)}{2}t^2 \bigl((u_2/a_p)^2 + (u_1/a_p)^2\bigr)  \leq  2 |t|^p + o(t^2).
$$
Dividing by $t^2$, we get that
$$
\limsup_{t \to 0} \frac{p\bigl(u_1/a_p + u_2/a_p\bigr)}{t} + \frac{p(p-1)}{2}\bigl((u_2/a_p)^2 + (u_1/a_p)^2\bigr) \leq 0.
$$
Taking into account that $t$ in the above $\limsup$ takes both negative and positive values, this implies that $u_1 = u_2 = 0$, that is, $Te_2 = 0$ and so $T$ is not an isomorphism.

On the other hand, suppose that $1<p<2$ and let $q$ be the number satisfying $1/p+1/q=1$. Then, we show that  there is no isomorphism $T\in \mathcal{L}(X)$ with $\|T\|=1$ such that $T(x_p) = e_1$. Indeed, we fix an operator $T\in \mathcal{L}(X)$ with $\|T\|=1$ such that $T(x_p) = e_1$. Then $T^*(e_1)=x_q$. The previous argument shows that $T^*e_2 =0$. So $T$ is not an isomorphism.
\end{proof}

Now we are ready to prove Theorem~\ref{theorem:noLp-microtransitive}.

\begin{proof}[Proof of Theorem~\ref{theorem:noLp-microtransitive}]
As we mentioned earlier, it is shown in \cite{AMS} that the Hilbertian norm is micro-transitive. Suppose that the norm of $L_p(\mu)$ is uniformly micro-semitransitive. Then it is uniformly smooth by Proposition~\ref{prop2}. So it is one-dimensional or $1<p<\infty$.
We may then assume that $1<p<\infty$ and that the dimension of $L_p(\mu)$ is greater than one. Then the result follows from Remark~\ref{Remark-one-complemented} and Example~\ref{example-ellp2-direct}.
\end{proof}

Our next result shows that our micro-transitivity properties transfer from the norm of a Banach space to the norm of its dual space.

\begin{proposition}\label{prop-appendix-duality}
A Banach space is uniformly micro-semitransitive (respectively, micro-transitive) if and only if its dual is.
\end{proposition}

\begin{proof} It suffices to check the ``only if'' part since by Corollary~\ref{corollary-quasi-transitive=>unif-smooth} uniformly micro-semitransitive spaces are reflexive.

If the norm of $X$ is uniformly micro-semitransitive, then $X$ is both uniformly smooth and uniformly convex, by Corollary \ref{corollary-quasi-transitive=>unif-smooth}. Hence, $X^*$ is  uniformly convex and uniformly smooth.
Let $\eps \longmapsto\beta(\eps)$ be a function witnessing the norm of $X$ to be uniformly micro-semitransitive. Fix $\eps\in (0,1)$. Given $x^*, y^*$ in $S_{X^*}$ with $$\|x^* - y^*\|<\delta_X(\beta(\eps)),$$ there exist  $x, y$ in $S_X$ satisfying that $x^*(x) =y^*(y)=1$.
Then, $$|x^*(x)- y^*(x)|<\delta_X(\beta(\eps)),\qquad \text{so}\ \ \re y^*(x) > 1-\delta_X(\beta(\eps))$$ and thus,
\[
\left\| \frac{x+y}{2}\right\|\geq \re y^*\left(\frac{x+y}{2}\right)> 1-\delta_X(\beta(\eps)).
\]
This means that $\|x-y\|<\beta(\eps)$. The uniform micro-semitransitivity of the norm of $X$ provides us with $S\in \mathcal{L}(X)$ satisfying that $\|S\|=1$ and
\[
Sx =  y, \qquad  \text{and} \qquad \|S-\Id\|<\eps.
\]
Then, we have $\|S^*\|= 1$ and $\|S^* - \Id\|<\eps$. As $[S^*y^*](x) = y^*(Sx)=y^*(y)=1$, we get that $S^*y^* =x^*$ by the (uniform) smoothness of $X$ recalling Corollary \ref{corollary-quasi-transitive=>unif-smooth}. This gives that the norm of $X^*$ is uniformly micro-semitransitive.

In the case of micro-transitive norms, if the operator $S$ in the preceding proof is a surjective isometry, then so is $S^*$,  giving thus the micro-transitivity of the norm of $X^*$.
\end{proof}

The proposition above shows that Corollary~\ref{cor-charact-BPBp-BPBpp} also holds for the dual space $X^*$ as well as $X$ if the space is micro-transitive. It is worth to note that using this fact, Example~\ref{example-ellp2-direct} can be proved by Corollary~\ref{cor-charact-BPBp-BPBpp}.e and Proposition~\ref{prop-appendix-duality} since the modulus of convexity of $\ell_p^{(2)}$ is equivalent to $\eps^p$ for $p>2$.

We do not know any examples of Banach spaces whose norm is micro-transitive other than Hilbert spaces, and it may be possible that this property actually implies the space to be Hilbertian. Even more, this can be also the case for Banach space with uniformly micro-semitransitive norm. Note that the ``moreover'' part of Corollary~\ref{cor-charact-BPBp-BPBpp} shows that finite dimensional 1-complemented subspaces of uniformly micro-semitransitive spaces have ``the right'' moduli of uniform convexity.
The next result also goes in this line.

\begin{proposition}
Let $X$ be a uniformly micro-semitransitive Banach space. Suppose further that there is a function $\beta\colon(0,2)\longrightarrow \R^+$ such that the norm of all two-dimensional subspaces of $X$ and all two-dimensional quotients of $X$ is uniformly micro-semitransitive witnessed by this function. Then,  there is a constant $C>0$ such that
\[
\delta_{X}(\eps) \geq C \eps^2\quad \text{and}  \quad \delta_{X^*}(\eps) \geq C \eps^2 \qquad (0<\eps<2).
\]
Therefore, $X$ has type $2$ and cotype $2$ and so, it is isomorphic to a Hilbert space.
\end{proposition}

\begin{proof}
Observe that both $X$ and $X^*$ are uniformly convex and uniformly smooth. Let $Y$ be a two dimensional subspace of $X$. Then $Y$ is uniformly convex with a modulus of convexity smaller than or equal to $\delta_X$ and its norm is uniformly micro-semitransitive witnessed by the function $\beta$. Moreover, the Banach-Mazur distance from $Y$ to $\ell_2^{(2)}$ is bounded (see \cite{Jo}). Therefore, Corollary \ref{cor-charact-BPBp-BPBpp}.e gives us that there is a constant $C>0$, not depending on $Y$, such that
\[
\delta_{Y}(\eps) \geq C\, \eps^2 \qquad (0<\eps<2).
\]
Since the above inequality holds for every two-dimensional subspace $Y$ of $X$, we get that $\delta_{X}(\eps) \geq C \eps^2$ and so, $X$ has cotype 2 \cite{Pi}.

Next, $X^*$ is uniformly convex and uniformly smooth, its norm is uniformly micro-semitransitive by Proposition \ref{prop-appendix-duality}, and the norm of all of its two-dimensional subspaces is uniformly micro-semitransitive witnessed the common function $\beta$ by Proposition \ref{prop-appendix-duality} and the hypothesis on the quotients of $X$. Therefore, the same argument above can be applied to all two-dimensional subspaces of $X^*$ to get that $X^*$ is uniformly convex of power type $2$, so $X$ is uniformly smooth of power type $2$. Therefore, $X$ has type $2$ \cite{Pi}. Finally, Kwapien's theorem \cite{Kw} says that $X$ is isomorphic to a Hilbert space.
\end{proof}

\begin{remark-final}
The recent paper \cite{BecerraRodriguez2019} considers quantitative, weak versions of uniform micro-semitransitivity and contains a stronger version of Corollary~\ref{corollary-quasi-transitive=>unif-smooth}.
\end{remark-final}

\end{document}